\documentclass[11pt]{preprint}
\usepackage{graphicx,psfrag,epsfig}
\usepackage{amssymb,amsmath,amscd,amsthm,mathrsfs}
\usepackage{graphicx,psfrag,epsfig,xcolor}
\usepackage{bm,enumerate}
\usepackage{times}
\usepackage{graphicx}
\usepackage{mathtools}
\usepackage{tikz}
\usepackage{microtype}
\usepackage{comment}
\usepackage{dsfont}
\usepackage[makeroom]{cancel}

\usepackage[normalem]{ulem}
\usepackage{soul}
\soulregister{\cite}{7}
\soulregister{\ref}{7}
\soulregister{\eqref}{7}
\definecolor{lightyellow}{Hsb}{60,.5,1}
\sethlcolor{lightyellow}
\setuloverlap{0pt}

\usepackage[breaklinks]{hyperref}

\makeatletter
\@ifpackagelater{hyperref}{2012/05/28}{
  \definecolor{link1}{Hsb}{240,1,.75}
  \definecolor{link2}{Hsb}{240,1,.5}
  \hypersetup{final,colorlinks,allcolors=link1,citecolor=link2}
}{
  \hypersetup{final}
}
\makeatother

\newtheorem{theorem}{Theorem}[section]
\newtheorem*{theorem*}{Theorem}

\newtheorem{lemma}[theorem]{Lemma}
\newtheorem{proposition}[theorem]{Proposition}

\newtheorem{remark}[theorem]{Remark}
\newtheorem*{remark*}{Remark}

\newcommand{\ONE}{\mathds{1}}

\definecolor{darkred}{rgb}{0.7,0.1,0.1}

\definecolor{red}{rgb}{1,0,0}
\definecolor{green}{rgb}{0,1,0}
\definecolor{blue}{rgb}{0,0,1}
\definecolor{cyan}{rgb}{0,1,1}

\iftrue
\else
  \RequirePackage{geometry}
\fi

\newcommand{\R}{\mathbb{R}}

\def\eps{{\varepsilon}}

\def\cF{\mathcal{F}}

\def\cV{\mathcal{V}}

\def\cO{\mathcal{O}}

\def\cI{\mathcal{I}}
\def\cD{\mathcal{D}}

\def\cI{\mathcal{I}}

\def\e{\eps}

\def\P{\mathbf{P}}
\def\E{\mathbf{E}}

\DeclareBoldMathCommand{\one}{1}

\numberwithin{equation}{section}

\renewcommand{\epsilon}{\eps}

\renewcommand{\ge}{\geqslant}
\renewcommand{\geq}{\ge}
\renewcommand{\le}{\leqslant}
\renewcommand{\leq}{\le}

\newcommand{\D}{\mathcal{D}}

\def\cC{\mathcal{C}}

\renewcommand{\downarrow}{\to}

\begin{document}
\title{Malliavin calculus approach to long exit times from an unstable equilibrium}
\author{Yuri Bakhtin, Zsolt Pajor-Gyulai}
\institute{
  New York University, \email{bakhtin@cims.nyu.edu, zsolt@cims.nyu.edu}
}
\maketitle
\begin{abstract}
For a one-dimensional smooth vector field in a neighborhood of an unstable equilibrium, we consider the associated
dynamics perturbed by small noise. Using  Malliavin calculus tools, we obtain precise vanishing noise asymptotics
for the tail of the exit time and for the exit distribution conditioned on atypically long exits.
\end{abstract}

\section{Introduction}\label{sec:setting}

Exit problems for small random perturbations of random dynamical systems have been studied for several decades. The most celebrated asymptotic results in this direction are large deviation estimates of the Freidlin--Wentzell theory and their extensions, see the classical book~\cite{FW2012}.

There are situations when large deviation results are not sufficient for detailed analysis of the system's behavior. In particular, the analysis of noisy heteroclinic networks, i.e., systems with multiple unstable equilibria  connected to each other
by heteroclinic orbits, requires studying distributional scaling limit theorems for exit points and exit times. This approach allows for an iteration scheme that leads to a detailed description of typical diffusion paths on time scales logarithmic in noise magnitude $\eps>0$, see~\cite{Bak2010},\cite{Bak2011},\cite{AB2011}.  In those papers, the
results on the asymptotics of exits from neighborhoods of critical points extend the results of~\cite{Kifer1981} and~\cite{Day95}
 where
the leading deterministic logarithmic term for the exit time $\tau^\eps$ and the leading random correction to the logarithmic term were  computed. Namely, for a class of initial conditions near the critical point
(or the associated stable manifold), it was established that
if $\lambda>0$ is the leading eigenvalue of the linearization of the dynamics near the critical point, then
\[
\tau^\eps=\frac{1}{\lambda}\log\frac{1}{\eps}+\theta_\eps,
\]
where random variables $\theta_\eps$ converge in distribution as $\eps\to0$. Moreover, the limiting distribution is 
nontrivial and has explicit representations. The associated distributions of exit locations were also
studied in~\cite{Eizenberg:MR749377} and~\cite{Bak2008}.

To extend the results of~\cite{Bak2010},\cite{Bak2011},\cite{AB2011} to longer time scales,
one needs to study rare events responsible for the unlikely transitions in the heteroclinic networks. One
such rare event can be described as withstanding the repulsion near an unstable critical point for an atypically long time. Thus, we need to study the asymptotics of the tail of $\tau_\eps$. The
best results in this direction, to the best of our knowledge, were established in~\cite{Mikami:MR1357028}, where a large deviation estimate in the form of logarithmic equivalence
was obtained:
\[
\lim_{\eps\to 0}\frac{\log \P\left(\tau^\eps>\frac{\alpha}{\lambda} \log \frac{1}{\eps}\right)}{\log \eps}=\alpha-1,\quad \alpha>1.
\]

In this paper, for the
one-dimensional case, we provide more delicate estimates proving that for all $\alpha>1$,
\begin{equation}\label{eq:main_res_informal}
\P\left(\tau^\eps>\frac{\alpha}{\lambda}\log\frac{1}{\eps}\right)=\Lambda\eps^{\alpha-1}(1+o(1)),\quad \eps\to 0,
\end{equation}
and computing the precise value of the constant~$\Lambda$, see our main result, Theorem~\ref{thm:main-theorem}, below. We give an explicit expression for~$\Lambda$ in terms of the starting point ranging through a neighborhood of the critical point that we describe. We also explicitly compute the limiting distribution of the exit location conditioned on the rare event of withstanding  the repulsion
for an atypically long time, and it turns out that it does not depend on the starting point within that neighborhood.

It is essential for our proof to study the asymptotic behavior of densities of certain auxiliary random variables. Besides the traditional methods of stochastic analysis, we use the Malliavin calculus tools. We believe
that our approach can be extended to higher dimensions, although the extension is not straightforward, and it will be addressed in a separate paper.

\medskip

Let us now introduce the setting more formally.
We will consider the family of stochastic differential equations
\begin{equation}\label{eq:SDE}
dX_{\eps}(t)=b\left(X_{\eps}(t)\right)dt+\eps\sigma\left(X_{\eps}(t)\right)dW(t),
\end{equation}
on a bounded interval $\cI=[q_-,q_+]\subseteq \mathbb{R}$, where the drift is given by a  vector field~$b\in\mathcal{C}^{2}(\R)$ and the random perturbation is given via a standard Brownian motion $W$
with respect to a filtration $(\cF_t)_{t\ge 0}$ defined on some probability space $(\Omega,\mathcal{F},\mathrm{P})$. The noise magnitude is given by a small parameter $\eps>0$ in front of the diffusion coefficient $\sigma\in\mathcal{C}^2(\R)$ which is assumed to satisfy $\sigma(0)>0$. Although we are interested only
in the evolution within $\cI$, we
can assume that $b$ and $\sigma$ are globally Lipschitz without changing the setting.

Standard results on stochastic differential equations (see e.g \cite{KS1991}) imply that for any starting location $X^{\eps}(0)\in\cI$, the equation~\eqref{eq:SDE} has a unique strong solution up to
\[
\tau_{\cI}^{\eps}=\inf\{t\geq 0: X_{\eps}(t)\in\partial \cI\},
\]
the exit time from $\cI$.

Let $(S^t)_{t\in\R}$ be the flow generated by the vector field $b$, i.e., $x(t)=S^tx_0$ is the solution of the autonomous ordinary differential equation
\begin{equation}\label{eq:deterministic-ODE}
\dot{x}(t)=b(x(t)),\qquad x(0)=x_0\in\R.
\end{equation}
We assume that there is a unique repelling zero of the vector field $b$. Without loss of generality we
place it at the origin. In other words, we assume that $b(0)=0$ and, for some $\lambda>0$ and $\eta\in\cC^2(\cI)$,
\begin{equation}\label{eq:linearizable}
b(x)=\lambda x+\eta(x)|x|^2,\qquad x\in\cI.
\end{equation}
Note that since the origin is the only zero of $b$ in the closed interval $\cI$, this assumption implies that for all $x\neq 0$, there is a uniquely defined
finite time $T(x)$ such that $S^{T(x)}\in\partial\cI$.


Under
the condition \eqref{eq:linearizable},  the map $f:\cI\to\R$ defined by
\begin{equation}\label{eq:coord-change-diffeo}
f(x)=\lim_{t\to\infty}e^{\lambda t} S^{-t}x=x-\int_0^\infty e^{\lambda s}\eta(S^{-s}x)|S^{-s}x|^2ds
\end{equation}
is a $\cC^2$-diffeomorphism, see \cite{Eizenberg:MR749377}. It preserves the order on $\R$, so $f(q_-)<0<f(q_+)$.

Our main result is the following.
\begin{theorem}\label{thm:main-theorem}
Consider $X_\e$ defined by \eqref{eq:SDE} with initial condition $X_\e(0)=\e x$ and let $K(\e)$ be a
function such that
\begin{equation}\label{eq:K-criterion}
\lim_{\e\downarrow 0}\e^{\beta}K(\e)=0,\qquad \forall \beta>0.
\end{equation}
Then, for all $\alpha>1$,
\begin{equation}\label{eq:main-thm-claim-1}
\lim_{\e\downarrow 0}\sup_{|x|\leq K(\e)}\left|\e^{-(\alpha-1)}\mathrm{P}\left(\tau_{\cI}^\e>\frac{\alpha}{\lambda}\log\e^{-1}\right)-
\sqrt{\frac{\lambda}{\pi}}\frac{e^{-\lambda\left(\frac{x}{\sigma(0)}\right)^2}}{\sigma(0)}\left(|f(q_+)|+|f(q_-)|\right)
\right|
=0
\end{equation}
and
\begin{equation}\label{eq:main-thm-claim-2}
\lim_{\e\downarrow 0}\sup_{|x|\leq K(\e)}\left|\mathrm{P}\left(X_\e\left(\tau_{\cI}^\e\right)=q_{\pm}\bigg|\tau_{\cI}^\e>\frac{\alpha}{\lambda}\log\e^{-1}\right)-\frac{|f(q_{\pm})|}{|f(q_+)|+|f(q_-)|}\right|=0.
\end{equation}
\end{theorem}

\begin{remark}\rm 
The second term in \eqref{eq:main-thm-claim-1} plays the role of $\Lambda$ in \eqref{eq:main_res_informal}. Note that the coefficient in front of $|f(q_+)|+|f(q_-)|$
is the centered Gaussian density with variance~$\sigma^2(0)/(2\lambda)$ evaluated at~$x$.
\end{remark}

One could approach Theorem~\ref{thm:main-theorem} 
using that the distribution of $X_\e$ conditioned to stay in the bounded domain for an atypically long time,  approaches the quasi-stationary distribution 
exponentially fast, see~\cite{Champagnat2016}. 
 However, our situation is more subtle since both, the time scale and the system, depend on~$\eps$. 
So, instead of attempting to appeal to the general quasi-stationary distribution theory, we adopt the following plan: we change the coordinates by conjugating the drift to a linear one; then we use Duhamel's principle to represent the solution and express the event of interest $\left\{\tau^\eps_{\cI}>\frac{\alpha}{\lambda}\log\frac{1}{\eps}\right\}$ in terms of a random factor in the resulting variation of constants
formula; the convergence of that factor in distribution is known and has been employed in the literature
cited above, but the desired result concerns unlikely events, and we need much stronger regularity, namely, uniform convergence of the associated probability densities. We use Malliavin
calculus to study these densities. In fact, the Malliavin calculus approach works only for small values of $\alpha$, and to extend it to the longer time scales we need to invoke an additional recursive scheme that can be seen as studying the quasi-stationary distribution since each step is performed under the no-exit
conditioning.

In this program, the density estimates are based on a well-known formula for the density in terms of the Malliavin derivative and the divergence operator, see Proposition~\ref{prop:Nua-dens}. 
This formula has been used to derive upper bounds on densities, see~\cite[Section 2.1.1]{Nualart2006}, but in general it is viewed as not very useful, see, e.g., \cite{Nourdin-Viens:MR2556018}, where a
replacement formula is suggested. However, in our context, Proposition~\ref{prop:Nua-dens} turns out to
be very efficient.

{\bf Acknowledgment.} We are grateful to Hong-Bin Chen who found a gap in our proof and suggested a fix that we are using in the present version of the paper.
Yuri Bakhtin gratefully acknowledges partial support from NSF via grant
DMS-1460595. 

\section{Proof of Theorem \ref{thm:main-theorem}}
We will study the system in a small neighborhood of the origin and  after the process has escaped this small neighborhood. 

Let us start with the first part. The diffeomorphism 
$f:\cI\to\R$ introduced in \eqref{eq:coord-change-diffeo} and its inverse $g=f^{-1}$ 
provide a conjugation between the flow~$(S^t)$ and a linear flow:
\begin{equation}\label{eq:conjugation}
f(S^tx)=e^{\lambda t}f(x),\qquad\textrm{or}\qquad f'(x)b(x)=\lambda f(x).
\end{equation}
Note that the integrand in \eqref{eq:coord-change-diffeo} is quadratic when $x$ is close to zero and thus we have $f(0)=0$ and $f'(0)=1$. Outside of $\cI$, we define $f$ so that $f'$ and $f''$ are bounded.

Let $Y_\e(t)=f(X_\e(t))$ for times prior to the escape from $\cI$.
It\^o's formula and \eqref{eq:conjugation} then imply that this process satisfies the stochastic differential equation
\begin{equation}\label{eq:lin-SDE}
dY_\e(t)=\lambda Y_\e(t)dt+\e\tilde\sigma(Y_\e(t))dW(t)+\frac{\e^2}{2}h(Y_\e(t))dt
\end{equation}
for $t<\tau_{\cI}^\e$, where $\tilde\sigma(y)=f'(g(y))\sigma(g(y))$ and $h(y)=f''(g(y))\sigma^2(g(y))$.
 Due to boundedness of $f'$ and $f''$, $\tilde\sigma$ and $h$ are also bounded.

Let us choose $R>0$ sufficiently small to ensure that  $\mathcal{V}=g\left([-R,R]\right)\subseteq \cI$
and $\tilde\sigma(x)> 0$ for all $x\in[-R,R]$. The folowing result describes the behavior of  $\tau_\cV^\e$, the exit time from~$\mathcal{V}$. 

\begin{theorem}\label{thm:linear}
Let $Y_\e(0)=\e x$, where $|x|\leq K(\e)$ with $K(\e)$ satisfying \eqref{eq:K-criterion}.
There is a family of random variables~$(M_\e)$ such that for all $\eps>0$, $\P(M_\e=0)=0$, 
\begin{equation}\label{eq:linear-exit-rep}
\tau_\cV^\e=\frac{1}{\lambda}\log\frac{R}{\e}-\frac{1}{\lambda}\log|M_\e|,
\end{equation}
\begin{equation}\label{eq:linear-exit-direction-claim}
\P\Big(\{Y_\e(\tau_\cV^\e)=\pm R\}\triangle\{\pm M_\e>0\}\Big)=0,
\end{equation}
and, if $a(\e)$ is a monotone function such that $a(\e)\sim c\e^{\theta}$, $\eps\to 0$, for some  $c,\theta>0$, then the following estimate holds:
\begin{equation}\label{eq:lin-thm-M_e-claim}
\sup_{|x|\leq K(\e)}\left|\P\Big(0< \pm M_\e\leq a(\e)\Big)-\sqrt{\frac{\lambda}{\pi}}
\frac{e^{-\lambda\left(\frac{x}{\sigma(0)}\right)^2}}{\sigma(0)}a(\e)\right|=o\left(a(\e)\right).
\end{equation}
\end{theorem}

We give the proof of Theorem~\ref{thm:linear} in Section~\ref{sec:linear}.

After exit from $\cV$, the deterministic dynamics dominates the evolution, which is captured by the following standard large deviation estimates.

\begin{proposition}\label{prop:LDPestim}
Let $X_{\e}(0)=g(\pm R)$. There are constants $c_1,c_2,\e_0>0$ such that
\begin{equation}\label{eq:LDP-for-time}
\P\left(\left|\tau_\cI^\e-T\left(g(\pm R)\right)\right|>\e N\right)\leq c_1e^{-c_2N^2},\quad \e<\e_0,\ N\geq 1, 
\end{equation}
and
\[
\P\left(X_\e(\tau_\cI^\e)=q_{\pm}\right)\geq 1-c_1e^{-c_2/\e^{2}},\quad \eps<\eps_0.
\]
\end{proposition}

\noindent\textit{Proof of Theorem \ref{thm:main-theorem}.}
~Combining $\tau_\cI^\e=\tau_\cV^\e+(\tau_\cI^\e-\tau_\cV^\e)$ with \eqref{eq:linear-exit-rep}, \eqref{eq:LDP-for-time}, and the strong Markov property yields the representation
\begin{equation}\label{eq:tau-I-repr}
\tau_\cI^\e=\frac{1}{\lambda}\log\e^{-1}+C-\frac{1}{\lambda}\log|M_\e|+\theta_\e,\qquad C=\frac{1}{\lambda}\log R+ T\left(X_\eps(\tau_\cV^\e)\right),
\end{equation}
where $\theta_\e$ is a random variable such that 
\begin{equation}\label{eq:theta-small}
\P(|\theta_\e|>\e N)\leq c_1e^{-c_2N^2},\quad N\ge 1,
\end{equation}
for some $c_1,c_2>0$. Although $R$ appears in the definition of $C$, one can easily show that $C$ does not, in fact, depend on the choice of $R$. Identities~\eqref{eq:tau-I-repr} and \eqref{eq:linear-exit-direction-claim} imply
\begin{align}\label{eq:left-or-right-decomp}
\P\left(\tau_\cI^\e\geq\frac{\alpha}{\lambda}\log\e^{-1}\right)&=\P\left(|M_\e|\leq e^{\lambda(C+\theta_\e)}\e^{\alpha-1}\right)=I_-+I_+,
\end{align}
where
\begin{equation}\label{eq:Ipm-CRpm}
I_{\pm}=\P\left(0<\pm M_\e\leq e^{\lambda(C^\pm+\theta_\e)}\e^{\alpha-1}\right),\qquad C^\pm=\frac{1}{\lambda}\log R+T\left(g(\pm R)\right).
\end{equation}
Let us simplify the definition of $C^{\pm}$. Since $T(g(\pm R))$ equals the time it takes
for the linear flow to travel between $R$ and $f(q_{\pm})$, i.e., $\frac{1}{\lambda}\log\frac{|f(q_\pm)|}{R}$, we see that
\begin{equation}
\label{eq:simpler-C}
C^{\pm}=\frac{1}{\lambda}\log |f(q_{\pm})|.
\end{equation}
We can write
\begin{align*}
I_{\pm}\leq \P\left(0<\pm M_\e\leq e^{\lambda(C^{\pm}+\e^{\beta})}\e^{\alpha-1}\right)+\P\left(|\theta_\e|> \e^{\beta}\right),
\end{align*}
where we choose $\beta\in (0,1)$. The second term decays exponentially fast a $\e\to0$ by \eqref{eq:theta-small} and we may apply \eqref{eq:lin-thm-M_e-claim} to the first term with $a(\e)=e^{\lambda(C^{\pm}+\e^{\beta})}\e^{\alpha-1}$ to conclude
\begin{equation}\label{eq:Ipm-upper-bound}
I_{\pm}\leq \sqrt{\frac{\lambda}{\pi}}\frac{e^{\lambda C^\pm}\e^{\alpha-1}}{\sigma(0)}e^{-\lambda\left(\frac{x}{\sigma(0)}\right)^2}\left(1+\mathcal{O}\big(\e^\beta\big)\right)+o(\e^{\alpha-1}),
\end{equation}
with the error term being uniform over $|x|\leq K(\e)$. Note that invoking \eqref{eq:lin-thm-M_e-claim} was justified by
\[
Y_\e(0)=f(\e x)=\e x +o(\e x)=\e x(1+o(1)).
\]
The analogous lower bound follows similarly:
\begin{align}
\label{eq:Ipm-lower-bound}I_{\pm}&\geq \P\left(0<\pm M_\e\leq e^{\lambda(C^{\pm}-\e^{\beta})}\e^{\alpha-1};~|\theta_\e|\leq\e^\beta\right)\\
\notag&\geq \P\left(0<\pm M_\e\leq e^{\lambda(C^{\pm}-\e^{\beta})}\e^{\alpha-1}\right)-\P\left(|\theta_\e|\geq \e^{\beta}\right)\\
\notag&=\sqrt{\frac{\lambda}{\pi}}\frac{e^{\lambda C^\pm}\e^{\alpha-1}}{\sigma(0)}e^{-\lambda\left(\frac{x}{\sigma(0)}\right)^2}\left(1+\mathcal{O}\big(\e^\beta\big)\right)+o(\e^{\alpha-1}).
\end{align}
Combining \eqref{eq:Ipm-upper-bound},~\eqref{eq:Ipm-lower-bound}, and~\eqref{eq:left-or-right-decomp}, we obtain
\[
\e^{-(\alpha-1)}\P\left(\tau_\cI^\e\geq\frac{\alpha}{\lambda}\log\e^{-1}\right)=\sqrt{\frac{\lambda}{\pi}}\,\frac{e^{\lambda C^-}+e^{\lambda C^+}}{\sigma(0)}e^{-\lambda\left(\frac{x}{\sigma(0)}\right)^2}+o(1),
\]
with the error term being uniform over $|x|\leq K(\e)$.
Using \eqref{eq:simpler-C}, we complete the proof of \eqref{eq:main-thm-claim-1}.

To prove \eqref{eq:main-thm-claim-2}, we write
\begin{align*}
\P\Big(X_\e(\tau_\cI^\e)=q_\pm;&~\lambda\tau_\cI^\e>(1-\alpha)\log\e^{-1}\Big)\\
&=\P\left(X_\e(\tau_\cI^\e)=q_\pm;|M_\e|\leq e^{\lambda(C+\theta_\e)}\e^{\alpha-1}\right)\\
&=\P\left( X_\e(\tau_{\cV}^\e)=g(\pm R);|M_\e|\leq e^{\lambda(C+\theta_\e)}\e^{\alpha-1}\right)+\mathcal{O}\left(e^{-c_2/\e^2}\right)\\
&=\P\left(0<\pm M_\e\leq e^{\lambda(C^{\pm}+\theta_\e)}\e^{\alpha-1}\right)+\mathcal{O}\left(e^{-c_2/\e^2}\right)
=I_{\pm}+\mathcal{O}\left(e^{-c_2/\e^2}\right).
\end{align*}
Here the first equality is due to \eqref{eq:tau-I-repr}, the second one is due to Proposition~\ref{prop:LDPestim}, while third one holds by~\eqref{eq:linear-exit-direction-claim}. Therefore,
\[
\mathrm{P}\left(X_\e\left(\tau_{\cI}^\e\right)=q_{\pm}\bigg|\tau_{\cI}>\frac{\alpha}{\lambda}\log\e^{-1}\right)=\frac{I_\pm+\mathcal{O}\left(e^{-c_2/\e^2}\right)}{I_-+I_+}
\longrightarrow
\frac{e^{\lambda C^{\pm}}}{e^{\lambda C^-}+e^{\lambda C^+}},\quad \eps\to0.
\]
and the proof is finished by \eqref{eq:simpler-C}--\eqref{eq:Ipm-lower-bound}. 
\section{The linear equation}
\label{sec:linear}
The goal of this section is to prove Theorem \ref{thm:linear}. To this end, we apply Duhamel's formula to the stochastic differential equation \eqref{eq:lin-SDE} to obtain
\begin{equation}\label{eq:Y-M-corresp}
Y_\e(t)=\e e^{\lambda t}M_\e(t),\qquad M_\e(t)=x+U_\e(t)+V_\e(t),
\end{equation}
where we used $Y_\e(0)=\e x$ and introduced
\[
U_\e(t)=\int_0^te^{-\lambda s}\tilde\sigma(Y_\e(s))dW(s),\qquad V_\e(t)=\frac{\e}{2}\int_0^te^{-\lambda s}h(Y_\e(s))ds.
\]
Note that there is $C>0$ such that for all $t\ge 0$ and $x\in\R$,
\begin{equation}
\label{eq:contributions-to-M}
\langle U_\eps\rangle_t<C,\qquad |V_\eps(t)|\le C\eps.  
\end{equation}
By the definition of $\tau_\cV^\e$, we have
\begin{equation}
\label{eq:basic-eq-on-exit}
R=|Y_\e(\tau_\cV^\e)|=\e e^{\lambda\tau_\cV^\e}|M_\e(\tau_\cV^\e)|,
\end{equation}
which we can rearrange to obtain
\[
\tau_\cV^\e=\frac{1}{\lambda}\log\frac{R}{\e}-\frac{1}{\lambda}\log|M_\e(\tau_\cV^\e)|.
\]
This establishes \eqref{eq:linear-exit-rep} once we set $M_\e=M_\e(\tau_\cV^\e)$. Note immediately that
\[
\P(M_\e=0)=\P(\tau_\cI^\e=\infty)=0
\] 
by the uniform ellipticity of $\tilde\sigma$. Moreover, the sign of $Y_\e(t)$ coincides with the sign of $M_\e(t)$ for all $t>0$ and thus \eqref{eq:linear-exit-direction-claim} is verified as well. Therefore it remains to prove \eqref{eq:lin-thm-M_e-claim}, and the rest of the section is dedicated to this goal.

The strategy of the proof is the following. We first replace the stopping time $\tau_{\cV}^\e$ with the deterministic time
\begin{equation}
\label{eq:deterministic-time}
T_\e=\frac{1}{\lambda}\log\frac{R}{\e}-\frac{1}{\lambda}\log a(\e),
\end{equation}
where we recall that $a(\e)\sim c\e^{\theta}$, which turns out to be a good substitute on the rare events that we seek to study. Next, we use tools from Malliavin calculus to establish that for an appropriately chosen shorter fixed time $T_\e'\leq T_\e$, the random variable $M_\e(T_\e')$ has a density around zero that converges to the density of a centered Gaussian random variable with variance $(2\lambda)^{-1}\sigma^2(0)$, which establishes the desired estimate for $M_\e(T_\e')$. Finally, we use the Markov property to iteratively extend this conclusion to $M_\e(T_\e)$. The following well-known exponential martingale inequality (see, e.g., Problem~12.10 in~\cite{Bass:MR2856623}) will be useful many times.

\begin{lemma}\label{lem:exp-marting-ineq}
Let $M(t)$ be a martingale with quadratic variation process $\langle M\rangle_t$. Then
\[
\mathrm{P}\left(\sup_{t\geq 0}|M(t)|\geq a; \langle M\rangle_{\infty}\leq b\right)\leq 2e^{-\frac{a^2}{2b}}
\]
for any $a,b>0$. In particular, due to~\eqref{eq:contributions-to-M},
\[
\mathrm{P}\left(\sup_{t\geq 0}|U_\e(t)|\geq a\right)\leq 2e^{-Ca^2},\qquad \E\left[\sup_{t\geq 0}|U_\e(t)|^p\right]\leq C_{p},\quad p\ge 1,
\]
for some constants $C,C_{p}<\infty$.
\end{lemma}

We start by showing that $M_\eps(T_\e)$ is a good substitute for $M_\eps=M_\eps(\tau^\e_{\cV})$ on the set $|M_\e|\leq a(\e)$. 
When working with deterministic times, we do not restrict the dynamics to times before the exit from $\cI$, so we use our global Lipschitzness assumption to ensure the existence of the strong solution. From now on, we use 
the same letter $C$ to denote various positive constants.


\begin{lemma}\label{lem:random-to-det-time}
For every $\gamma\in(0,1)$ and $\gamma'>0$, there is $\e_0>0$ such that
\[
\sup_{|x|\leq K(\e)}\P\Big(\left|M_\e-M_\e(T_\e)\right|\geq a(\e)\e^\gamma;~|M_\e|\leq a(\e)\Big)<\e^{\gamma'}, \quad \e<\e_0.
\]
\end{lemma}

\begin{proof}
Observe that $\{|M_\e|\leq a(\e)\}=\{\tau_\cV^\e\geq T_\e\}$  by~\eqref{eq:basic-eq-on-exit} 
and~\eqref{eq:deterministic-time}. 
 This allows us to write
\begin{align*}
\P&\left(|U_\e(\tau_\cV^\e)-U_{\e}(T_\e)|\geq\frac{a(\e)\e^\gamma}{2};~|M_\e|\leq a(\e)\right)\\
&~~~~~~~~~~~~~~~~
=\P\left(\left|\int_{T_\e}^{\tau_\cV^{\e}}e^{-\lambda t}\tilde{\sigma}(Y_\e(s))dW(s)\right|\geq \frac{a(\e)\e^\gamma}{2};~\tau_\cV^\e\geq T_\e\right)\\
&~~~~~~~~~~~~~~~~\leq\P\left(\sup_{t\geq T_\e}\left|\int_{T_\e}^te^{-\lambda t}\tilde{\sigma}(Y_\e\left(s)\right))dW(s)\right|\geq \frac{a(\e)\e^\gamma}{2}\right)\\
&~~~~~~~~~~~~~~~~\leq 2\exp\left\{-C\frac{a^2(\eps) \eps^{2\gamma}}{e^{-2\lambda T_\eps}} \right\}=
2\exp\left\{-C\frac{a^2(\eps) \eps^{2\gamma}}{e^{-2(\log\frac{R}{\e}-\log a(\e)) }} \right\}
= 2\exp\left\{-C\frac{\eps^{2\gamma}}{ \eps^2 } \right\}\leq\frac{\e^{\gamma'}}{2},
\end{align*}
for any $\gamma'>0$, and sufficiently small $\eps$, where we use the boundedness of $\tilde{\sigma}$ and Lemma \ref{lem:exp-marting-ineq} in the last inequality. At the same time, on $\{\tau_\cV^{\e}\geq T_\e\}$, we have
\[
|V_\e(\tau_\cV^\e)-V_{\e}(T_\e)|\leq \frac{\e}{2\lambda}\sup_{t\leq\tau_\cV^\e}|a(Y_{\e}(t))|e^{-\lambda T_\e}\leq \frac{\e^2a(\e)\|h\|_\infty }{2\lambda R},
\]
where we used that $h$ is bounded. Therefore, the triangle inequality and a simple union bound gives
\begin{align*}
\P\Big(\left|M_\e-M_\e(T_\e)\right|\geq & a(\e)\e^\gamma;~|M_\e|\leq a(\e)\Big)\\
\leq & \P\left(|U_\e(\tau_\cV^\e)-U_{\e}(T_\e)|\geq \frac{a(\e)\e^\gamma}{2};~|M_\e|\leq a(\e)\right)\\
&+ \P\left(|V_\e(\tau_\cV^\e)-V_{\e}(T_\e)|\geq \frac{a(\e)\e^\gamma}{2};~|M_\e|\leq a(\e)\right)\leq\e^{\gamma'}
\end{align*}
provided $\gamma\in (0,1)$ and $\e$ is sufficiently small.
\end{proof}

Our next goal is to show that $M_\e(T_\e)$ satisfies the desired small ball asymptotics.  The next proposition is the key technical result of this paper and will be proved along other
useful results in Section \ref{sec:mall-prop-proof} with Malliavin calculus  tools. 
\begin{proposition}\label{prop:main-malliavin-lemma}
Assume $Y_\e(0)=\e x$ and let $T_\e'>0$ be a 
function of $\e$ 
such that
\begin{equation}\label{eq:T-prime-cond}
\frac{\lambda T_\e'}{\log\e^{-1}}\in\left[1-\frac{c}{\log\e^{-1}},2-\kappa\right]
\end{equation}
for some $c,\kappa>0$ and for all sufficiently small $\e$ . The random variable $M_\e'=M_\e(T_\e')$ has a continuous, bounded density $p_{\e}^{x}(z)$. Moreover, for every $K(\e)$ satisfying \eqref{eq:K-criterion},
\begin{equation}
\lim_{\e\downarrow 0}\sup_{|x|\leq K(\e)}\sup_{z\in\R}\Big(|p_{\e}^{x}(z)-p^{x}(z)|e^{|x-z|}\Big)=0,
\label{eq:rate-of-convergence-of-densities}
\end{equation}
where $p^x$ is the density of a Gaussian random variable with mean $x$ and variance $(2\lambda)^{-1}\sigma^2(0)$:
\[
p^{x}(z)=\sqrt{\frac{\lambda}{\pi}}\frac{1}{\sigma(0)}e^{-\lambda\left(\frac{z-x}{\sigma(0)}\right)^2}.
\]
\end{proposition}

In the proof of the theorem on the outcome of the linear evolution, we will use the Markov property many times, so it is convenient to  denote by $\P_y$ the joint distribution of the driving Wiener process $W$ and the process $Y_\e$ with initial point $Y_\e(0)=y$.

\noindent\textit{Proof of Theorem \ref{thm:linear}.}~ As observed in the beginning of this section, it only remains to show that~$M_\e$ satisfies \eqref{eq:lin-thm-M_e-claim}. Let us derive the theorem from the following
statement that we will prove later: if $b(\e)=o\left(a(\e)\right)$ and $K(\e)$ satisfies \eqref{eq:K-criterion}, then 
$P(\eps,x)=\P_{\e x}\left(b(\e)\leq M_\e(T_\e)\leq a(\e)\right)$ satisfies
\begin{equation}\label{eq:M-at-full-T}
\lim_{\eps\to 0}\sup_{|x|\leq K(\e)}\left|\frac{P(\eps,x)}{a(\eps)}-p^x(0)\right|=0.
\end{equation}
Let us choose $\gamma'$ large enough such that 
$\eps^{\gamma'}=o(a(\eps))$, $\eps\to 0$.
Lemma~\ref{lem:random-to-det-time} and~\eqref{eq:M-at-full-T} imply that if $\gamma\in(0,1)$, then
\begin{align}
\label{eq:upper-bound-M_e-probs}
\P_{\e x}\Big(0< M_\e\leq a(\e)\Big)&\leq \e^{\gamma'}+\P_{\e x}\Big(0<M_\e\leq a(\e);~\left| M_\e-M_\e(T_\e)\right|<a(\e)\e^\gamma\Big)
\\
\notag
&\leq \e^{\gamma'}+\P_{\e x}\Big(-a(\e)\e^\gamma\leq M_\e(T_\e)\leq a(\e)(1+\e^\gamma)\Big)= p^x(0)a(\e)+o\left(a(\e)\right).
\end{align}
Similarly, 
\begin{align}
\label{eq:lower-bound-M_e-probs}
\P_{\e x}
\Big(0< M_\e\leq a(\e)\Big)
&\geq \P_{\e x}\Big(0<M_\e\leq a(\e);~\left| M_\e-M_\e(T_\e)\right|<a(\e)\e^\gamma\Big)
\\
\notag&\geq \P_{\e x}\Big(a(\e)\e^\gamma\leq M_\e(T_\e)\leq a(\e)(1-\e^{\gamma});~\left| M_\e-M_\e(T_\e)\right|<a(\e)\e^\gamma\Big)
\\
\notag&= \P_{\e x}\Big(a(\e)\e^\gamma\leq M_\e(T_\e)\leq a(\e)(1-\e^{\gamma})\Big)-\Delta(\eps,x),
\\
\notag&= 
p^x(0)a(\e)+o\left(a(\e)\right)
-\Delta(\eps,x),
\end{align}
where
\begin{align*}
\Delta(\eps,x)&=\P_{\e x}\Big(a(\e)\e^\gamma\leq M_\e(T_\e)\leq a(\e)(1-\e^{\gamma});~\left| M_\e-M_\e(T_\e)\right|\ge a(\e)\e^\gamma\Big)
\le\Delta_1(\eps,x)+\Delta_2(\eps,x),
\end{align*}
with
\begin{align*}
\Delta_1(\eps,x)&= \P_{\e x}\Big(\tau_\cV^\e\ge T_\e;~| M_\e-M_\e(T_\e)|\ge a(\e)\e^\gamma\Big),\\
\Delta_2(\eps,x)&= \P_{\e x}\Big(\tau_\cV^\e\leq T_\e;~|M_\e(T_\e)|\leq a(\e)(1-\e^\gamma)\Big).
\end{align*}
Due to~\eqref{eq:upper-bound-M_e-probs} and~\eqref{eq:lower-bound-M_e-probs}, the desired relation \eqref{eq:lin-thm-M_e-claim} for the positive
sign follows from 
\begin{equation}
\label{eq:Deltas-small}
\Delta_i(\eps,x)=o(a(\eps)),\quad i=1,2, 
\end{equation}
uniformly in $x$.
Since $\{\tau_\cV^\e \ge T_\e\}=\{|M_\e| \le a(\e)\}$, 
Lemma~\ref{lem:random-to-det-time} immediately implies that $\Delta_1(\eps,x)<\eps^{\gamma'}$ for any $\gamma'>0$
and sufficiently small $\eps$.
To estimate~$\Delta_2(\eps,x)$, we note that $|M_\e(T_\e)|\leq a(\e)(1-\e^\gamma)$ implies
\[
|Y_{\e
}(T_\e)|=\e e^{\lambda T_\e} |M_\e(T_\e)|\leq R(1-\e^\gamma),
\]
and applying the strong Markov property to the process $Y_\eps$ after $\tau_\cV^\e$, we obtain
\[
\Delta_2(\eps,x)\leq \max_{z=\pm R}\P_z\left(\inf_{t\in[0,T_\e]}|Y_{\e}(t)|\leq R(1-\e^{\gamma})\right).
\]
Duhamel's formula~\eqref{eq:Y-M-corresp} and the reverse triangle inequality imply
\[
|Y_\e(t)|\geq e^{\lambda t}\left||Y_\e(0)|-\e |U_\e(t)|-\e |V_\e(t)|\right|\geq |Y_\e(0)|-\e|U_\e(t)|-\e|V_\e(t)|,
\]
and, combining the last two displays, \eqref{eq:contributions-to-M}, and Lemma~\ref{lem:exp-marting-ineq}, we obtain
\[
\Delta_2(\eps,x)\leq \P\Big(\sup_{t\in [0,T_\e]}|U_\e(t)|\geq R\e^{-(1-\gamma)}-C\Big)\leq 3e^{-\frac{C}{\e^{2(1-\gamma)}}}< \e^{\gamma'},
\]
for any $\gamma<1$ and sufficiently small $\e$.  Thus, \eqref{eq:Deltas-small} is verified, finishing the proof of \eqref{eq:lin-thm-M_e-claim} for the positive sign. The result for $-M_\e$ is proved the same way.

%
%
%
%

\medskip

It remains to prove \eqref{eq:M-at-full-T}. 
If $\theta\in(0,1)$, then  $T_\e=T'_\eps$ satisfies \eqref{eq:T-prime-cond}, so 
\eqref{eq:M-at-full-T} immediately follows from Lemma \ref{lem:random-to-det-time} and Proposition \ref{prop:main-malliavin-lemma}. If $\theta\geq 1$, then we cannot apply Proposition \ref{prop:main-malliavin-lemma} directly,
so our strategy will be to extend it to longer times using an iterative procedure based on the Markov property.
Namely, we let $N=\lfloor\theta\rfloor+1$ and
consider the shorter time  $T'_\eps=T_\eps/N$. There is $\kappa>0$ 
such that for sufficiently small
$\eps$,
\[
\alpha_{\eps}:=\frac{\lambda T'_\eps}{\log \eps^{-1}}=\frac{1+\theta-\zeta_\eps}{\lfloor\theta\rfloor+1}\in \left(1-\frac{C}{\log\e^{-1}},2-\kappa\right),
\]
where $\zeta_\eps=\cO\left(1/\log\e^{-1}\right)$. We recall that $C>0$ stands for a constant that may change on every appearance.
In fact, observe also that the last display and $N\in\mathbb{N}$ imply
\begin{equation}\label{eq:N-interval}
2\leq N\leq\frac{\theta}{\alpha_\e}+1+\frac{C}{\log\e^{-1}}
\end{equation}
for sufficiently small $\e$. For simpler notation, let us use the abbreviations
\[
T_{\e,k}=k T_{\e}',\qquad M_{\e,k}=M_\e(T_{\e,k}),\qquad Y_{\e,k}=Y_\e(T_{\e,k}),\qquad \mathcal{F}_{\e,k}=\mathcal{F}_{T_{\e,k}},\qquad k=1,\dots, N.
\]

We will show by induction that for each $\eps>0$ there is a sequence $(H_{\e,k})_{k=0}^\infty$ of centered Gaussian random variables independent of $\mathcal{F}_{\e,k}$, with variance 



\begin{equation}\label{eq:H-variance}
\E H_{\e,k}^2=\frac{\sigma^2(0)}{2\lambda}\frac{1-\e^{2(N-k)\alpha_\e}}{1-\e^{2\alpha_\e}},
\end{equation}
and such that 
\begin{equation}\label{eq:to-be-proved-by-ind}
P(\eps,x)=\P_{\e x}\left(b(\e)\leq M_{\e,k}+\e^{k\alpha_\e} H_{\e,k}\leq a(\e)\right)+o(a(\e)),\quad \e\downarrow 0, 
\end{equation}
for all $k=0,\dots,N$. Using \eqref{eq:to-be-proved-by-ind} with $k=0$ finishes the proof of \eqref{eq:M-at-full-T} as $M_{\e,0}=x$ and the random variable $H_{\e,0}$ is centered Gaussian with variance converging to $(2\lambda)^{-1}\sigma^2(0)$ as $\e\downarrow 0$ by \eqref{eq:H-variance}.

Now we proceed with the proof of \eqref{eq:to-be-proved-by-ind} by first noting that the case $k=N$ is trivial as $H_{\e,N}=0$. Let us assume that \eqref{eq:to-be-proved-by-ind} holds for some $k=1,\dots,N$ and show that it therefore holds for $k-1$ as well. Due to the  relation $Y_\e(t)=\e e^{\lambda t}M_\e(t)$, the Markov property allows us to write:
\begin{multline*}
\P_{\e x}\left(b(\e)\leq M_{\e,k}+\e^{k\alpha_\e} H_{\e,k}\leq a(\e)\right)=\E_{\e x}\left[\P_{\e x}\left(b(\e)\leq M_{\e,k}+\e^{k\alpha_\e}H_{\e,k}\leq a(\e)|\mathcal{F}_{T_{\e,{k-1}}}\right)\right]\\
=\int_{-\infty}^{\infty} \P_{\e x}\left(b(\e)\leq M_{\e,k}+\e^{k\alpha_\e}H_{\e,k}\leq a(\e)\bigg|M_{\e,k-1}=z\right)r_{x}^{(k-1)}(dz),
\end{multline*}
where we introduced the measure
\[
r_{x}^{(k-1)}(dz)=\P_{\e x}\left(M_{\e,k-1}\in dz\right).
\]
Since $Y_{\e,k}=\e e^{\lambda T_{\e,k}}M_{\e,k}=\e^{1-k\alpha_\e}M_{\e,k}$, for all $k=0,\dots,N$, the integrand equals
\begin{align*}
\P_{\e x}\left(\e^{1-k\alpha_\e}b(\e)\leq Y_{\e,k}+\e H_{\e,k}\leq \e^{1-k\alpha_\e}a(\e)\bigg|\ Y_{\e,k-1}=\e^{1-(k-1)\alpha_\e}z\right)\\
=\P_{\e^{1-(k-1)\alpha_\e}z}\left(\e^{1-k\alpha_\e}b(\e)\leq Y_{\e,1}+\e H_{\e,k}\leq \e^{1-k\alpha_\e}a(\e)\right),
\end{align*}
where we used the Markov property and the independence of $H_{\e,k}$ and $\mathcal{F}_{\e,k}$ in the last step. Note  that $Y_{\e,1}$ is independent of $H_{\e,k}$. Combining the last three displays and the change of variables $z\to \e^{(k-1)\alpha_\e}z$ gives
\begin{align}
\label{eq:M(2T)ba}\P_{\e x}&\left(b(\e)\leq M_{\e,k}+\e^{k\alpha_\e}  H_{\e,k}\leq a(\e)\right)\\
\notag &=\int_{-\infty}^\infty \P_{\e z}\left(\e^{1-k\alpha_\e}b(\e)\leq Y_{\e,1}+\e H_{\e,k}\leq \e^{1-k\alpha_\e}a(\e)\right)r_{x}^{(k-1)}(\e^{(k-1)\alpha_\e}dz)
\end{align}

Let us now prove that there are $c_1, c_2>0$ such that for any $L>0$, we have
\begin{align}\label{eq:fun-fact-1}
\sup_{|z|\geq L}\P_{\e z}\left(\e^{1-k\alpha_\e}b(\e)\leq Y_{\e,1}+\e H_{\e,k}\leq \e^{1-k\alpha_\e}a(\e)\right)&\leq c_1 e^{-c_2 L^2}.
\end{align}
Note that \eqref{eq:N-interval} implies
\[
(k-1)\alpha_\e\leq(N-1)\alpha_\e\leq \theta+\frac{C}{\log\e^{-1}},
\]
so, using that $\e^{C/\log\e^{-1}}=e^{-C}=const>0$, we have
\[
\e^{1-\alpha_\e}\geq\e^{1-k\alpha_\e+C/\log\e^{-1}}a(\e)\geq C\e^{1-k\alpha_\e}a(\e),
\]
 which implies
\begin{align*}
|Y_{\e,1}+\e H_{\e,k}|&=|Y_{\e}(T_\e')+\e H_{\e,k}|\geq\e e^{\lambda T_\e'}\left(|z|-|U_{\e}(T_\e')|-|V_\e(T_\e')|\right)-\e|H_{\e,k}|\\
&\geq \e^{1-\alpha_\e}\left(|z|-|U_{\e}(T_\e')|-|V_\e(T_\e')|-\e^{\alpha_\e}|H_{\e,k}|\right)\\
&\geq C_1\e^{1-k\alpha_\e}a(\e)\left(|z|-|U_{\e}(T_\e')|-\eps C_2-\e^{\alpha_\e}|H_{\e,k}|\right).
\end{align*}
Since $b(\e)=o\left(a(\e)\right)$, the left-hand side of \eqref{eq:fun-fact-1} can be thus bounded above by
\[
\sup_{|z|\geq L}\P_{\e z}\left(|Y_{\e,1}+\e H_{\e,k}|\leq \e^{1-k\alpha_\e}a(\e)\right)
\leq\P\Big(|H_{\e,k}|\geq L\Big)+\P\Big(|U_\e(T_\e')|\geq  L(1-\e^{\alpha_\e})-\frac{1}{C_1}
-\eps C_2 \Big),
\]
and \eqref{eq:fun-fact-1}  follows from the standard Gaussian tail bound, \eqref{eq:H-variance}, and
 Lemma \ref{lem:exp-marting-ineq}.

Conditioning on $H_{\e,k}$, using the independence of $M_\e(T_\e')$ and $H_{\e,k}$,
and using the existence of density of $M_\eps(T'_\eps)$ guaranteed by Proposition~\ref{prop:main-malliavin-lemma}, 
 we obtain
\begin{align*}
\P&_{\e z}\left(\e^{1-k\alpha_\e}b(\e)\leq Y_{\e,1}+\e H_{\e,k}\leq \e^{1-k\alpha_\e}a(\e)\right)\\
&=\P_{\e z}\left(\e^{-(k-1)\alpha_\e}b(\e)\leq M_{\e}(T_\e')+\e^{\alpha_\e}H_{\e,k}\leq \e^{-(k-1)\alpha_\e}a(\e)\right)= \E\int_{b(\e)\e^{-(k-1)\alpha_\e}-\e^{\alpha_\e}H_{\e,k}}^{a(\e)\e^{-(k-1)\alpha_\e}-\e^{\alpha_\e}H_{\e,k}}p_\e^{z}(u)du.
\end{align*}
 This, along with the induction hypothesis 
 \eqref{eq:to-be-proved-by-ind}, relations \eqref{eq:M(2T)ba}--\eqref{eq:fun-fact-1}, and Fubini's theorem, gives
\begin{align}\label{eq:exp+bigint}
P(\eps,x)=\E\int_{b(\e)\e^{-(k-1)\alpha_\e}-\e^{\alpha_\e}H_{\e,k}}^{a(\e)\e^{-(k-1)\alpha_\e}-\e^{\alpha_\e}H_{\e,k}}\int_{-L}^{L}p_{\e}^{z}(u)~r_{x}^{(k-1)}\left(\e^{(k-1)\alpha_\e}dz\right)du+\mathcal{O}\left(e^{-\lambda L^2}\right)
+o(a(\eps)).
\end{align}
To use the uniform convergence statement of Proposition \ref{prop:main-malliavin-lemma}, while simultaneously making the error term decay sufficiently fast, we choose
\[
L=L(\e)=K(\eps)=\sqrt{\frac{2\theta}{\lambda}\log\e^{-1}}.
\]
Let us estimate the error we make if we replace $p_{\e}^{z}$ by $p^{z}$ in~\eqref{eq:exp+bigint}.
We take a random variable $\Psi$ 
independent of $M_{\e,k-1}$ and $H_{\e,k}$,
with a centered Laplace density $f_\Psi(x)=e^{-|x|}/2$, $x\in\R$, 
 and apply Proposition~\ref{prop:main-malliavin-lemma} to bound
$|p_{\e}^{z}(u)-p^z(u)|$ using $f_\Psi$. This allows us to
write
\begin{align*}
&\E\int_{b(\e)\e^{-(k-1)\alpha_\e}-\e^{\alpha_\e}H_{\e,k}}^{a(\e)\e^{-(k-1)\alpha_\e}-\e^{\alpha_\e}H_{\e,k}}\int_{-L}^{L}|p_{\e}^{z}(u)-p^z(u)|~r_{x}^{(k-1)}\left(\e^{(k-1)\alpha_\e}dz\right)du
\\
=&o(1)\,\P_{\e x }\left\{\eps^{-(k-1)\alpha_\eps}|M_{\e,k-1}|\le L;\  \eps^{-(k-1)\alpha_\eps}M_{\e,k-1} +\Psi+\eps^{\alpha_\eps}H_{\e,k}\in \eps^{-(k-1)\alpha_\eps}[b(\eps), a(\eps)] \right\}
\\
=&o(1)\,\P_{\e x }\left\{M_{\e,k-1} +\eps^{(k-1)\alpha_\eps}(\Psi+\eps^{\alpha_\eps}H_{\e,k})\in [b(\eps), a(\eps)] \right\}=o(a(\eps)),
\end{align*}
where in the last identity we used
the independence of $M_{\e,k-1}$  and $\Psi+\eps^{\alpha_\eps}H_{\e,k}$
and the fact that random variables $M_{\eps,k}$, $k=1,\ldots,N$, $\eps>0$, have uniformly bounded
densities. The latter is a direct consequence of Lemma~\ref{lem:mall-deriv-estim} from  Section~\ref{sec:mall-prop-proof} and 
Proposition~2.1.2 of~\cite{Nualart2006} estimating the density of a Wiener
functional $F$ in terms of moments of $\|\cD F\|_H^{-1}$,  and 
$ \|\cD^2 F \|_{H\otimes H}$.

This allows us to write \eqref{eq:exp+bigint} as 
\begin{align*}
P(\eps,x)=\E\int_{b(\e)\e^{-(k-1)\alpha_\e}-\e^{\alpha_\e}H_{\e,k}}^{a(\e)\e^{-(k-1)\alpha_\e}-\e^{\alpha_\e}H_{\e,k}} \int_{-\infty}^{\infty}p^z(u)\P_{\e x}\left( \e^{-(k-1)\alpha_\e}M_{\e,k-1}\in dz\right)du+o(a(\e)),\\
\end{align*}
where we also used the decay rate of $p^z(u)$ at $\infty$ to restore the integration domain to the entire line.

Recalling that $p^z(u)=p^0(u-z)$, we change variables: $u'=\e^{(k-1)\alpha_\e}u$, $z'=\e^{(k-1)\alpha_\e}z$, and express the integral on the right-hand side of the last display as
\begin{align*}
\e^{-(k-1)\alpha_\e}&\E\int_{b(\e)-\e^{k\alpha_\e}H_{\e,k}}^{a(\e)-\e^{k\alpha_\e}H_{\e,k}} \int_{-\infty}^{\infty}p^{0}\left(\e^{-(k-1)\alpha_\e}(u'-\e^{(k-1)\alpha_\e}z)\right)\P_{\e x}\left( M_{\e,k-1}\in\e^{(k-1)\alpha_\e}dz\right)du'\\
&=\e^{-(k-1)\alpha_\e}\E\int_{b(\e)-\e^{k\alpha_\e}H_{\e,k}}^{a(\e)-\e^{k\alpha_\e}H_{\e,k}} \int_{-\infty}^{\infty}p^{0}\left(\e^{-(k-1)\alpha_\e}(u'-z')\right)\P_{\e x}\left( M_{\e,k-1}\in dz'\right)du'
\\
&=\P_{\e x}\left(b(\e)\leq M_{\e,k-1}+\e^{(k-1)\alpha_\e}\xi+\e^{k\alpha_\e}H_{\e,k}\leq a(\e)\right)
\\
&=\P_{\e x}\left(b(\e)\leq M_{\e,k-1}+\e^{(k-1)\alpha_\e}(\xi+\e^{\alpha_\e}H_{\e,k})\leq a(\e)\right),
\end{align*}
where $\xi$ is a centered Gaussian random variable independent of both $M_{\e,k-1}$ and $H_{\e,k}$ with variance $(2\lambda)^{-1}\sigma^2(0)$. This completes the proof of the induction step once one verifies, by a straightforward computation, that $H_{\e,k-1}:=\xi+\e^{\alpha_\e}H_{\e,k}$ has the desired variance given by \eqref{eq:H-variance}.
\qed

\begin{remark}\rm
In the second part of the previous proof, the random variables $\e^{(k-1)\alpha_\e}\xi_k$ with
\[
\xi_k=H_{\e,k-1}-\e^{\alpha_\e}H_{\e,k}
\]
essentially correspond to the distributional limit of the integral $I_k=\int_{(k-1)T_\e'}^{kT_\e'}e^{-\lambda t}\tilde\sigma(Y_\e(s))dW(s)$ conditioned on $Y_\e$ not having exited a neighborhood of the origin of size $\mathcal{O}(\e)$.
\end{remark}

\section{Malliavin calculus tools and the proof of Proposition \ref{prop:main-malliavin-lemma}}\label{sec:mall-prop-proof}
In this section we turn to the proof of Proposition \ref{prop:main-malliavin-lemma} by showing that $M'_\e:=M_\e(T'_\e)$ has a bounded continuous density with the desired asymptotics. We are going to employ Malliavin calculus tools, using~\cite{Nualart2006} as a basic reference. 

Let us first recall some basic notions. Let $H=L^2(\R_+)$ be the separable Hilbert space 
of square integrable functions on the line and let $B:H\mapsto L^2(\Omega)$ be the \emph{isonormal Gaussian process} on $H$ given by
\[
h\mapsto B(h)=\int_0^{\infty}h(s)dW(s).
\]

Let $L^2(\Omega;H)$ be the set of square integrable $H$-valued random variables with norm $\|u\|_{L^2(\Omega;H)}=\sqrt{\E\|u\|_H^2}$. The \emph{Malliavin derivative} operator $\cD:L^2(\Omega)\mapsto L^2(\Omega;H)$ is defined on random variables of the form 
\begin{equation}\label{eq:mall-deriv-simple-func}
F=f\left(B(h_1),\dots,B(h_n)\right),\qquad n\geq 1,
\end{equation}
by the formula
\[
\cD F=\sum_{i=1}^n\partial_{x_i}f\left(B(h_1),\dots,B(h_n)\right)h_i.
\]
It is extended to a closed operator under the graph norm
\[
\|F\|_{1,2}=\sqrt{\|F\|_2^2+\|\cD F\|^2_{2;H}},
\]
with domain $\mathrm{D}_{1,2}$, where we adopt the notation $\|F\|_2\equiv \|F\|_{L^2(\Omega)}=\sqrt{\E F^2}$ and $\|u\|_{2;H}=\|u\|_{L^2(\Omega;H)}$.

This construction can be extended to Hilbert space valued random variables producing 
a closed operator $\cD:L^2(\Omega;H)\mapsto L^2(\Omega;H\otimes H)$  
under the graph-norm 
\[
\|u\|_{1,2;H}=\sqrt{\|u\|_{2;H}^2+\|\cD u\|_{2;H\otimes H}^2},
\]
with domain $\mathrm{D}_{1,2;H}$,
where $\|\cD u\|_{2;H\otimes H}=\|u\|_{L^2(\Omega;H\otimes H)}$.
The second-order Malliavin derivative is the composition of the operators described above, a closed operator $\cD^2:L^2(\Omega)\to L^2(\Omega,H\otimes H)$.

The \emph{divergence operator} or \emph{Skorokhod integral} $\delta: L^2(\Omega;H)\mapsto L^2(\Omega)$ is defined as the $L^2(\Omega; H)$ adjoint of $\cD$, that is,
\[
\E\langle \cD F, u\rangle_{L^2(\Omega;H)}=\E\left[ F\delta(u)\right],\qquad u\in \mathrm{Dom}~\delta,
\]
where $\mathrm{Dom}~\delta$ is defined as the set of those $u\in L^2(\Omega;H)$ such that the left-hand side of the indicated identity defines a bounded functional as a function of $F$. In particular, $H$ can be embedded naturally into $\mathrm{Dom}~\delta$ and $\delta(h)=B(h)$ for any $h\in H$.

\begin{proposition}[\cite{Nualart2006}, Proposition 2.1.1]\label{prop:Nua-dens}
Let $F\in L^2(\Omega)\in \mathrm{D}_{1,2}$ and assume that 
\begin{equation}\label{eq:D-per-D-in-dom-delta}
\frac{\cD F}{\|\cD F\|_{H}^2}\in\mathrm{Dom}~\delta.
\end{equation}
Then $F$ has a bounded continuous density given by
\begin{equation}\label{eq:density-formula-abstract}
p(z)=\E\left[\ONE_{\{F>z\}}\delta\left(\frac{\cD F}{\|\cD F\|_{H}^2}\right)\right].
\end{equation}
\end{proposition}

We are going to use this result with $F=M_\e'$ under the measure $\P=\P_{\e x}$. We will write $p^x_{\eps}(z)$ for its density given by \eqref{eq:density-formula-abstract}. We are going to compare $M_\e'$ with the limiting random variable
\begin{equation}
I= x+\sigma(0)\int_0^\infty e^{-\lambda t}dW(t),
\label{eq:I}
\end{equation}
which is a centered Gaussian random variable with density $p^{x}(z)$. In what follows, $C$ will denote a positive finite constant, independent of $\e$, $x$, and $t$, which may change on each appearance.

\begin{lemma}\label{lem:M-and-I-close}
Let $Y_\e(0)=\e x$ with $|x|\leq K(\e)$, where $K(\e)$ satisfies \eqref{eq:K-criterion}. Then for any $\gamma\in(0,2)$ and sufficiently small $\eps$, we have
\begin{equation*}
\E |I-M'_\eps|^2\leq C\e^{\gamma}.
\end{equation*}
\end{lemma}

\begin{proof}
We start by estimating
\begin{align}\label{eq:EYsquare-bound}
\E\left[Y_\e(t)\right]^2&=\e^2 e^{2\lambda t}\E\left[M_\e(t)\right]^2=\e^2 e^{2\lambda t}\E\left[x+U_\e(t)+V_\e(t)\right]^2 \\
\notag&\leq C\e^2 e^{2\lambda t}\left(|x|^2+\E \left[U_\e(t)\right]^2+ \E\left[V_\e(t)\right]^2\right)\leq C\e^2e^{2\lambda t} (1+|x|^2),
\end{align}
where the first inequality is due to the elementary $(a+b+c)^2\leq 3(a^2+b^2+c^2)$, while the second one follows from Lemma \ref{lem:exp-marting-ineq} and~\eqref{eq:contributions-to-M}.

Next, we write
\begin{equation}\label{eq:I-M-decomp}
I-M_\e'=\int_0^{T_\e'}e^{-\lambda t}\left[\sigma(0)-\tilde\sigma(Y_\e(t))\right]dW(t)+ \sigma(0)\int_{T_\e'}^\infty e^{-\lambda t}dW(t)-V_\e(T_\e')=J_1(\eps)+J_2(\eps)+J_3(\eps). 
\end{equation}
Using the boundedness of $\tilde\sigma'(x)$ and the identity $\tilde\sigma(0)=\sigma(0)$ (which is due to $f'(0)=1$), we can write
\[
\E J_1^2(\eps)\le C\int_0^{T_\e'}e^{-2\lambda t}\E\left[Y_\e(t)\right]^2 dt\leq C\e^2(1+K^2(\e))T_\e'\leq C\e^{\gamma}
\]
for any $\gamma\in (0,2)$ and small $\eps$, where the first inequality is due to \eqref{eq:EYsquare-bound} and $|x|\leq K(\e)$.
By \eqref{eq:T-prime-cond}, 
\[
\E J_2^2(\eps)\le \sigma^2(0)\E\left[\int_{T_\e'}^\infty e^{-\lambda t}dW(t)\right]^2=\frac{\sigma^2(0)}{2\lambda}e^{-2\lambda T_\e'}\leq C \e^2
\]
for sufficiently small $\e$. Also, $\E J_3^2(\eps)\leq C\e^2$ by \eqref{eq:contributions-to-M}. 
 Combining \eqref{eq:I-M-decomp} with the the elementary inequality $(a+b+c)^2\leq 3(a^2+b^2+c^2)$ and the above estimates, we finish the proof.
\end{proof}

To use Proposition \ref{prop:Nua-dens}, we need to control the divergence and have certain estimates on the Malliavin derivatives of $M_\e'$. This is the content of the next two results. 

\begin{proposition}[\cite{Nualart2006}, Proposition 1.5.8] \label{prop:div-estim}
There is  $C>0$ such that
\[
\|\delta(u)\|_2\leq C\left(\|\E u\|_H+\|\cD u\|_{2;H\otimes H}\right),
\]
where it is implicit that the finiteness of the right-hand side implies $u\in\mathrm{Dom}~\delta$.
\end{proposition}

Let us recall that 
$
T_{\e,k}=k T_{\e}'$, and $M_{\e,k}=M_\e(T_{\e,k}),$ $k=1,\dots, N,$ were introduced in Section~\ref{sec:linear}.

\begin{lemma}\label{lem:mall-deriv-estim}
Let $Y_\e(0)=\e x$ and let $K(\e)$ satisfy \eqref{eq:K-criterion}. Then for every
$k=1,\ldots,N$,
\begin{equation}
\label{eq:convergence_of-DM}
\lim_{\e\downarrow 0}\sup_{|x|\leq K(\e)}\|\cD M_{\e,k}-\cD I\|_{2;H}=0,
\end{equation}
and, for every $m\geq 1$, 
\begin{align}
\label{eq:negative-moment-bounded}
\limsup_{\e\to 0}\sup_{|x|\leq K(\e)}\E \|\cD M_{\e,k} \|_H^{-m}<\infty,\\
\label{eq:D-second-vanishes}
 \lim_{\e\downarrow 0}\sup_{|x|\leq K(\e)}\E\|\cD^2 M_{\e,k}\|_{2;H\otimes H}^m=0.
\end{align}
\end{lemma}

\medskip

The proof of Lemma \ref{lem:mall-deriv-estim} will be given in Section \ref{sec:Mall-deriv-est}.
 We proceed now with the proof of Proposition \ref{prop:main-malliavin-lemma}.

\medskip

\noindent\textit{Proof of Proposition \ref{prop:main-malliavin-lemma}.}~
Let us first assume that $F=M_\e'$ satisfies \eqref{eq:D-per-D-in-dom-delta} and, as before, write $p^x_{\eps}(z)$ for its density. 
We are going to compare $M_\e'$ with the limiting random variable $I$ (introduced in~\eqref{eq:I}), which is a centered Gaussian random variable with density $p^{x}(z)$.

We start by introducing
\[
\Delta_\eps=\frac{\cD M'_\eps}{\|\cD M'_\eps\|_H^2}-\frac{\cD I}{\|\cD I\|_H^2}
\]
and writing
\begin{align}
\label{eq:approximating-density}
|p^x_{\eps}(z)-p^x(z)|=&\left|\E \left[ \ONE_{\{M'_\eps>z\}}\delta\left(\frac{\cD M'_\eps}{\|\cD M'_\eps\|_H^2}\right)-\ONE_{\{I>z\}}\delta\left(\frac{\cD I}{\|\cD I\|_H^2}\right)\right]\right|
\\
\notag
= &\left|\E \left[ \ONE_{\{M'_\eps>z\}}\delta(\Delta_\eps)\right]\right| +\left|\E \left[ \left(\ONE_{\{M'_\eps>z\}}-\ONE_{\{I>z\}}\right)\delta\left(\frac{\cD I}{\|\cD I\|_H^2}\right)\right]\right|
\\
\notag
\le&
\left[\P\Big(|U_\eps(T'_\eps)+V_\eps(T'_\eps)| > |z-x | \Big)\right]^{\frac{1}{2}}\left\|\delta(\Delta_\eps)\right\|_{2}+
\|\ONE_{\{M'_\eps>z\}}-\ONE_{\{I>z\}}\|_2 \left\| \delta\left(\frac{\cD I}{\|\cD I\|_H^2}\right)\right\|_2.
\notag
\end{align}
In the last step, besides  the Cauchy--Schwartz inequality, we used that 
(due to $\E\, \delta(\Delta_\eps)=0$)
\[\E\, \ONE_{\{U_\eps(T'_\eps)+V_\eps(T'_\eps)>z-x\}} \delta(\Delta)= - \E\, \ONE_{\{U_\eps(T'_\eps)+V_\eps(T'_\eps)\le z-x\}} \delta(\Delta).\]

Let us first deal with the second term on the right-hand side of \eqref{eq:approximating-density}. Observe
\begin{equation}\label{eq:DI-formulas}
\cD_t I=e^{-\lambda t}\sigma(0),\qquad \|\cD I\|_H^2=\frac{\sigma^2(0)}{2\lambda},
\end{equation}
so $\D I$ is deterministic and
\begin{equation}\label{eq:div-I-norm-bounded}
 \left\| \delta\left(\frac{\cD I}{\|\cD I\|_H^2}\right)\right\|_2=\frac{2\lambda}{\sigma^2(0)}\left\|\int_0^\infty e^{-\lambda t}\sigma(0)dW(t)\right\|_2=\frac{\sigma(0)}{\sqrt{2\lambda}}<\infty.
\end{equation}
The other factor can be estimated as follows. Assuming $z>x$, for any $\eta\in(0,1)$,
\begin{align}
\notag
\E\left[\ONE_{\{M'_\eps>z\}}-\ONE_{\{I>z\}}\right]^2&= \P(M'_\eps\le z<I\ \text{or}\ I\le z<M'_\eps )
\\&\le 
\P\left(|I-z|\le \eta\right)+\P\left(|I-M'_\eps|>\eta, I>z\right) +\P\left(|I-M'_\eps|>\eta, M'_\eps>z\right)
\notag
\\&
\le Ce^{-(z-x)}\eta+\left(\P(I>z)^{1/2}+\P(M'_\eps>z)^{1/2}\right)\P(|I-M'_\eps|>\eta)^{1/2}
\notag
\\& 
\le Ce^{-(z-x)}\left(\eta+\frac{(\E|I-M'_\eps|^2)^{1/2}}{\eta}\right).
\label{eq:diff-of-indicators}
\end{align}
We used the Cauchy--Schwarz inequality and a crude estimate of the Gaussian density of $I$ in the third line; in the last line, we used the Markov inequality and Gaussian tails of $I$ and $M_\eps$ (the latter is due to Lemma~\ref{lem:exp-marting-ineq}). Choosing $\eta=\e^{\gamma'}$ for any $\gamma'\in (0,1/2)$ and invoking Lemma~\ref{lem:M-and-I-close} with $\gamma=4\gamma'$ shows that the right-hand side of \eqref{eq:diff-of-indicators} is bounded by $Ce^{-|z-x|}\eps^{\gamma'}$.
If $z<x$, a similar estimate can be applied to
$\left[\ONE_{\{M'_\eps>z\}}-\ONE_{\{I>z\}}\right]^2 = \left[\ONE_{\{M'_\eps\le z\}}-\ONE_{\{I\le z\}}\right]^2$.

 This and \eqref{eq:div-I-norm-bounded} imply that the second term on the right-hand side of \eqref{eq:approximating-density} has the desired rate of decay.

Next we have to estimate the first term on the right-hand side of \eqref{eq:approximating-density}. 
We wish to apply Proposition~\ref{prop:div-estim} to the $H$-valued random variable $u_\e$ given by
\begin{equation}
\label{eq:u_eps}
u_\e(t)=\frac{\cD_tM'_\eps}{\|\cD M'_\eps\|_H^2}-\frac{\cD_tI}{\|\cD I\|_H^2}=\cD_tM'_\e\frac{\|\cD I\|_H^2-\|\cD M'_\e\|_H^2}{\|\cD M'_\e\|_H^2\|\cD I\|_H^2}+\frac{\cD_tM'_\e-\cD_t I}{\|\cD I\|_H^2}.
\end{equation}
The $H$-norm of $u_\e$ can be estimated using the triangle inequality and \eqref{eq:DI-formulas}:
\begin{align}
\notag
\|u_\e\|_H &\le C\Bigg(\frac{\|\cD I\|_H^2-\|\cD M'_\e\|_H^2}{\|\cD M'_\e\|_H}+\|\cD_tM'_\e-\cD_t I\|_H\Bigg)
\\
\notag &\le C\Bigg(\frac{\big(\|\cD I\|_H^2-\|\cD M'_\e\|_H\big)\big(\|\cD I\|_H^2+\|\cD M'_\e\|_H\big)}{\|\cD M'_\e\|_H}+\|\cD_tM'_\e-\cD_t I\|_H\Bigg)
\\&\leq C\|\cD M'_\e-\cD I\|_H\left[1+\frac{1}{\|\cD M_\e'\|_H}\right].
\label{eq:u_eps2}
\end{align}
Using Jensen's inequality, \eqref{eq:u_eps2}, and the Cauchy--Schwartz inequality yields
\begin{align}
\label{eq:norm-of-expect}\|\E u_\e\|_H\leq\E\|u_\e\|_H
&\leq C\left[\E\|\cD M'_\e-\cD I\|_H+\E\frac{\|\cD M'_\e-\cD I\|_H}{\|\cD M_\e'\|_H}\right]\\
\notag&\leq C\|\cD M'_\e-\cD I\|_{2;H}\left[1+\sqrt{\E\|\cD M'_\eps\|_H^{-2}}\right]\to 0
\end{align}
as $\e\downarrow 0$, where the convergence is due to Lemma \ref{lem:mall-deriv-estim}. 

To treat the second term given by Proposition \ref{prop:div-estim}, we observe that $\cD I$ being deteriministic implies $\cD^2 I=0$ and thus
\[
\cD_su_\e(t)=\cD_s\left[\frac{\cD_tM'_\eps}{\|\cD M'_\eps|_H^2}\right],
 \]
where the right-hand side is estimated by (2.5) in \cite{Nualart2006} implying 
\[
\|\cD u_\e\|_{H\otimes H}\leq\frac{3\|\cD^2M'_\e\|_{H\otimes H}}{\|\cD M'_\e\|_H^2},
\]
so we can use the Cauchy--Schwartz inequality to obtain
\begin{equation}\label{eq:Du-estimate}
\|\cD u_\e\|_{2; H\otimes H}\leq C\left(\E\left[\|\cD M'_\e\|_H^{-8}\right]\right)^{\frac{1}{4}}\left(\E\|\cD^2M'_\e\|_{H\otimes H}^4\right)^{\frac{1}{4}}\to 0
\end{equation}
as $\e\downarrow 0$ by Lemma \ref{lem:mall-deriv-estim}. Combining \eqref{eq:norm-of-expect} and \eqref{eq:Du-estimate} with Proposition \ref{prop:div-estim} gives
\[
\lim_{\e\downarrow 0}\left\|\delta\left(\frac{\cD M'_\eps}{\|\cD M'_\eps\|_H^2}-\frac{\cD I}{\|\cD I\|_H^2}\right)\right\|_{2}=0
\]
uniformly in $|x|\leq K(\e)$. A standard estimate of $\P\Big(|U_\eps(T'_\eps)+V_\eps(T'_\eps)| > |z-x | \Big)$ finishes the verification of \eqref{eq:rate-of-convergence-of-densities}.

To complete the proof, it remains to show that $\tilde{u}_\e=\cD M_\e'/\|\cD M_\e'\|_H^2\in\mathrm{Dom}~\delta$. Clearly, Jensen's inequality and Lemma \ref{lem:mall-deriv-estim} imply
\[
\|\E\tilde u_\e\|_H\leq \E\|\tilde u_\e\|_H=\E\|\cD M_\e\|_H^{-1}<\infty, 
\]
while $\cD I=0$ and \eqref{eq:Du-estimate} give us
\[
\|\cD\tilde u_\e\|_{2;H\otimes H}=\|\cD u_\e\|_{2;H\otimes H}<\infty.
\]
The desired conclusion now follows from Proposition \ref{prop:div-estim} and the last two displays.
\qed
\section{Proof of Lemma \ref{lem:mall-deriv-estim}}\label{sec:Mall-deriv-est}

In this section, we complete the proof of Theorem \ref{thm:main-theorem} by proving Lemma \ref{lem:mall-deriv-estim}. 
We will prove it only for $k=1$, i.e., for $M'_\e=M_{\e,1}$. The proof for $k=2,\ldots,N$ is completely the same.
As in the previous section, $C$ will denote a positive finite constant, independent of $\e$, $x$, and $t$, which may change on each appearance and we will use the elementary inequality $\left(\sum_{i=1}^na_i\right)^m\leq n^{m-1}\sum_{i=1}^na_i^m$, Fubini's theorem, and the boundedness of $\tilde\sigma$, $h$ and all their derivatives without further mention.

We start by recalling that the Malliavin derivative $\cD_tM_\e(u)$ satisfies a stochastic integral equation \cite[Theorem 2.2.1]{Nualart2006})
\begin{align}\label{eq:DtM-integral-eq}
\cD_t M_\e(u)&=e^{-\lambda t}\tilde{\sigma}(Y_\e(t))+\int_t^ue^{-\lambda s}\tilde{\sigma}'(Y_\e(s))\cD_tY_\e(s)dW(s)+\frac{\e}{2}\int_t^u e^{-\lambda s}h'(Y_\e(s))\cD_tY_\e(s)ds\\
\notag&=e^{-\lambda t}\tilde{\sigma}(Y_\e(t))+\e\int_t^u\tilde{\sigma}'(Y_\e(s))\cD_tM_\e'(s)dW(s) +\frac{\e^2}{2}\int_t^u h'(Y_\e(s))\cD_tM_\e'(s)ds
\end{align}
for $u\geq t$, while $\cD_tM_\e(u)=0$ otherwise. Here, we used~\eqref{eq:Y-M-corresp} and 
Proposition~1.3.8 from~\cite{Nualart2006} (see also formula (1.65) in this reference) in the first step, and \eqref{eq:Y-M-corresp} in the second one. This integral equation is equivalent to a stochastic differential equation in the variable $u$:
\begin{equation}
d\left[\cD_tM_\e(u)\right]=\cD_tM_\e(u) dZ_\e(u),\qquad \cD_tM_\e(t)=e^{-\lambda t}\tilde\sigma(Y_\e(t)),
\label{eq:linearized-dynamics} 
\end{equation}
where we introduced the semimartingale
\[
Z_\e(u)=\e\int_t^u\tilde{\sigma}'(Y_\e(s))dW(s) +\frac{\e^2}{2}\int_t^u h'(Y_\e(s))ds.
\]
It is well known (see, e.g.,~\cite[Section~5.6]{KS1991}) that the solution of the linear equation~\eqref{eq:linearized-dynamics} is given by the Dol\'eans-Dade exponential
\begin{align}\label{eq:explicit-formula-DM}
\mathcal{D}_tM_\e(u)&=e^{-\lambda t}\tilde\sigma(Y_\e(t))\exp\left(Z_{\e}(u)-\frac{\langle Z_\e\rangle_u}{2}\right)\\
\notag&=e^{-\lambda t}\tilde\sigma(Y_\e(t)) \exp\left(\e\int_t^u\tilde{\sigma}'(Y_\e(s))dW(s)+\frac{\e^2}{2}\int_t^u\left(h'(Y_\e(s))-(\tilde\sigma'(Y_\e(s))^2\right)ds\right).
\end{align}

\begin{lemma}\label{lem:DM-estim-rough}
For any $m\geq 1$, we have
\[
\sup_{u\in[0,T_\e']}\E|\D_tM_{\e}(u)|^{m}\leq C e^{-m\lambda t}.
\]
\end{lemma}

\begin{proof}
We can use the explicit formula \eqref{eq:explicit-formula-DM} to estimate
\begin{align*}
\E|\D_tM_{\e}(u)|^{m}&\leq C e^{-m\lambda t}\E\left[\exp\left(\e m\int_t^u\tilde{\sigma}'(Y_\e(s))dW(s)+\frac{\e^2m}{2}\int_t^u\left(h'(Y_\e(s))-(\tilde\sigma'(Y_\e(s))^2\right)ds\right)\right]\\
&\leq C e^{-m\lambda t}\E\left[\exp\left(\e m\int_t^u\tilde{\sigma}'(Y_\e(s))dW(s)-\frac{\e^2m^2}{2}\int_t^u\left(\tilde\sigma'(Y_\e(s)\right)^2ds\right)\right]=C e^{-m\lambda t},
\end{align*}
where  we used $\e^2(u-t)\leq \e^2T_\e'\to 0$, as $\e\downarrow 0$, in the second inequality to manipulate the Lebesgue integrals in the exponent, while the last step is due to the martingale property of the process under the expectation.
\end{proof}

Let us now prove the claims of  Lemma~\ref{lem:mall-deriv-estim} one by one.


\begin{proof}[Proof of~\eqref{eq:convergence_of-DM}]
Let $I(u)=\sigma(0)\int_0^ue^{-\lambda s}dW(s)$. Then
\[
\cD_tI(u)=\sigma(0)e^{-\lambda t} \ONE_{\{t\leq u\}}.
\]
By the triangle inequality, we have
\[
|\cD_t M_{\e}'-\cD_t I|\leq|\cD_t M_{\e}(T_\e')-\cD_t I(T_\e')|+e^{-\lambda t}\sigma(0)\ONE_{\{t> T_\e'\}},
\]
so
\begin{align}
\label{eq:DM-DI-2-norm}
\|\cD M_{\e}-\cD I\|_{2;H}^2&=\int_0^\infty\E|\cD_t M_{\e}-\cD_t I|^2dt
\\
\notag& 
\leq C \int_0^\infty\E|\cD_t M_{\e}(T_\e')-\cD_t I(T_\e')|^2dt+C\sigma^2(0) e^{-2\lambda T_\e'}.
\end{align}
The second term on the right-hand side converges to zero simply by $T_\e'\to\infty$. To estimate the first term, we use \eqref{eq:DtM-integral-eq} and the triangle inequality:
\begin{equation}\label{eq:yet-another-formula}
|\cD_t M_{\e}(T_\e')-\cD_t I(T_\e')|\leq\ONE_{\{t\leq T_\e'\}}Ce^{-\lambda t}|\tilde\sigma(Y_\e(t))-\sigma(0)|
+|R_{\e}(t)|,
\end{equation}
where
\[
R_\e(t)=\e\int_t^{T_\e'}\tilde{\sigma}'(Y_\e(s))\cD_tM_\e'(s)dW(s) +\frac{\e^2}{2}\int_t^{T_\e'} h'(Y_\e(s))\cD_tM_\e'(s)ds.
\]
The second moment of this latter quantity can be bounded as
\[
\E|R_\e(t)|^2\leq C\e^2(1+\e^2T_\e')\int_t^{T_\e}\E|\cD_tM_\e'(s)|^2ds\leq C\e^2e^{-2\lambda t},
\]
where the first inequality is due to Jensen's inequality, while we used Lemma \ref{lem:DM-estim-rough} and $\lim_{\e\downarrow 0}\e^2T_\e'=0$ in the second one. Combining this with \eqref{eq:yet-another-formula} and \eqref{eq:EYsquare-bound}, we obtain
\begin{align*}
\E|\cD_t M_{\e}(T_\e')-\cD_t I(T_\e')|^2&\leq C\big(\ONE_{\{t\leq T_\e'\}}e^{-2\lambda t}\E|Y_\e(t)|^2
+\E|R_{\e}(t)|^2\big)\\
&\leq C\Big(\e^2 (1+K^2(\e))\ONE_{\{t\leq T_\e'\}}+\e^2e^{-2\lambda t}\Big),
\end{align*}
and integrating with respect to $t$ gives us
\[
\int_0^\infty\E|\D_t M_{\e,\lambda}(T_\e')-\D_t I(T_\e')|^2dt\leq C\e^2T_\e'(1+K^2(\e))\to 0
\]
as $\e\downarrow 0$. Therefore the right-hand side of \eqref{eq:DM-DI-2-norm} converges to zero finishing the proof.
\end{proof}

%

\begin{proof}[Proof of~\eqref{eq:negative-moment-bounded}]
Applying Jensen's inequality for the convex function $x\mapsto |x|^{-m/2}$ and the measure on $[0,T_\e]$ with density
\[
\frac{e^{-2\lambda t}}{\int_0^{T_\e'}e^{-2\lambda s}ds}=\frac{2\lambda e^{-2\lambda t}}{1-e^{-2\lambda T_\e'}},
\]
we obtain
\begin{align*}
\E\|DM'_\e\|_H^{-m}&=\E\left(\int_0^{T_\e'}e^{-2\lambda t}e^{2\lambda t}|\cD_tM_\e'|^2dt\right)^{-m/2}
\\
&
\leq\left(\frac{2\lambda}{1-e^{-2\lambda T_\e'}}\right)^{\frac{m}{2}+1}\int_0^{T_\e'}e^{-2\lambda t}e^{-m\lambda t}\E|\cD_t M_\e'|^{-m}dt.
\end{align*}
Using the explicit formula \eqref{eq:explicit-formula-DM}, we can estimate
\begin{align*}
e^{-m\lambda t}\E|\cD_t M_\e'|^{-m}&\leq C\E \exp\left(-\e m\int_t^{T_\e'}\tilde{\sigma}'(Y_\e(s))dW(s)-\frac{\e^2m}{2}\int_t^{T_\e'}\left(h'(Y_\e(s))+(\tilde\sigma'(Y_\e(s))^2\right)ds\right)\\
&
\leq C\E\left[\exp\left(\e m\int_t^{T_\e'}\tilde{\sigma}'(Y_\e(s))dW(s)-\frac{\e^2m^2}{2}\int_t^{T_\e'}\left(\tilde\sigma'(Y_\e(s)\right)^2ds\right)\right]=C,
\end{align*}
where we used $\lim_{\e\downarrow 0}\e^2 T_\e'=0$ in the second inequality to manipulate the Lebesgue integrals, while the last step is due to  the martingale property of the exponential. Combining the last two displays finishes the proof.
\end{proof}

%

\begin{proof}[Proof of~\eqref{eq:D-second-vanishes}]
It is sufficient to prove the claim for $m\ge 2$ since convergence for $m\in[1,2)$ will follow from Lyapunov's inequality.
Taking the Malliavin derivative of both sides of \eqref{eq:DtM-integral-eq}
and using Proposition 1.3.8 in \cite{Nualart2006}, we obtain  that the second derivative $\cD^2_{r,t}M_\e'(u)$ satisfies the integral equation
\begin{align*}
\cD_{r,t}^2M_\e'(u)=&\e\tilde\sigma'(Y_\e(t))\cD_rM_\e'(t)+\e\tilde\sigma'(Y_\e(r)) \cD_tM_\e'(r)\\
&+\e^2\int_{r\vee t}^ue^{\lambda s}\tilde\sigma''(Y_\e(s))\cD_rM_\e'(s)\cD_tM_{\e}'s(s)dW(s)+\e\int_{r\vee t}^u \tilde\sigma'(Y_\e(s))\cD_{r,t}^2 M_\e(s)dW(s)\\
&+\frac{\e^3}{2}\int_{r\vee t}^u e^{\lambda s}h''(Y_\e(s))\cD_rM_\e'(s)\cD_tM_\e'(s)ds+\frac{\e^2}{2}\int_{r\vee t}^u h'(Y_\e(s))\cD_{r,t}^2M_\e'(s)ds,
\end{align*}
for $u\geq r\vee t$, while $\cD^2_{r,t}M_\e'(u)=0$ if $u<r\vee t$. 
Taking $m$-th  moments, we can apply the a priori bound in Lemma \ref{lem:DM-estim-rough}, the BDG inequality, and Jensen's inequality to estimate
\begin{align*}
\E|\cD_{r,t}^2M_\e'(u)|^m\leq& C\e^2\left(e^{-m\lambda t}+e^{-m\lambda r}\right)+C\e^{2m}(T_\e')^{\frac{m}{2}-1}\int_{r\vee t}^ue^{m\lambda s}\E\left|\cD_rM_\e'(s)\cD_tM_\e'(s)\right|^mds\\
&+C\e^m(T_\e')^{m-1}\int_{r\vee t}^u\E\left|\cD_{r,t}^2M_\e'(s)\right|^mds.
\end{align*}
The Cauchy--Schwartz inequality and Lemma \ref{lem:DM-estim-rough} again imply
\[
\E\left|\cD_rM_\e'(s)\cD_tM_\e'(s)\right|^m\leq \sqrt{\E|\cD_rM_\e'(s)|^{2m}\E|\cD_tM_\e'(s)|^{2m}}\leq C\sqrt{e^{-2m\lambda r}e^{-2m\lambda t}}= Ce^{-m\lambda(r+t)}.
\]
Combining the last two displays, we obtain
\begin{align*}
\E|\cD_{r,t}^2M_\e'(u)|^m \leq& C\e^m\left(e^{-m\lambda r}+e^{-m\lambda t}+\e^m(T_\e')^{\frac{m}{2}-1} e^{m\lambda (u-r-t)}\right)\\
&+C\e^m(T_\e')^{m-1}\int_{r\vee t}^u\E\left|\cD_{r,t}^mM_\e'(s)\right|^2ds.
\end{align*}
We recall Gronwall's inequality: if nonnegative functions $h, g, k$ satisfy
\[
h(u)\leq g(u)+\int_a^uk(s)h(s)ds,\quad u\ge a, 
\]
and $g$ is nondecreasing, then
\[
h(u)\leq g(u)e^{\int_a^uk(s)ds},\quad u\ge a.
\]
Using this with $h(u)=\E|\D_{r,t}^2M_\e'(u)|^m$, $a=r\vee t$,
\[
g(u)=C\e^m\left(e^{-m\lambda r}+e^{-m\lambda t}+\e^m(T_\e')^{\frac{m}{2}-1} e^{m\lambda (u-r-t)}\right),
\qquad
k(u)=C\e^m(T_\e')^{m-1},
\]
and setting $u=T_\e'$, we obtain
\begin{align*}
\E|\cD_{r,t}^2M_\e'|^m&\leq C\e^m\left(e^{-m\lambda r}+e^{-m\lambda t}+\e^m(T_\e')^{\frac{m}{2}-1} e^{m\lambda (T_\e'-r-t)}\right)e^{C\e^m(T_\e')^{m-1}(T_\e'-r\vee t)}\\
&\leq C\e^m\left(e^{-m\lambda r}+e^{-m\lambda t}+\e^m(T_\e')^{\frac{m}{2}-1} e^{m\lambda (T_\e'-r-t)}\right)e^{C\e^m(T_\e')^{m}}\\
&\leq C\left(\e^m \left(e^{-m\lambda r}+e^{-m\lambda t}\right)+\e^{2m}(T_\e')^{\frac{m}{2}-1}e^{m\lambda (T_\e'-r-t)}\right)\\
&\leq C\left(\e^m \left(e^{-m\lambda r}+e^{-m\lambda t}\right)+\e^{2m}\eps^{-\rho} e^{m\lambda (T_\e'-r-t)}\right)
\end{align*}
for any $\rho>0$ and sufficiently small $\e$, 
where we used $\lim_{\e\downarrow 0}\e^{a}(T_\e')^{b}= 0$ for all $a,b>0$. Using this along with Jensen's inequality, we obtain
\begin{align*}
\E\|\cD^2 M_\e'\|_{H\otimes H}^m&=\E\left(\int_0^{T_\e'}\int_0^{T_\e'}|\cD_{r,t}^2M_\e'|^2drdt\right)^{\frac{m}{2}}\leq (T_\e')^{m-2}\int_0^{T_\e'}\int_0^{T_\e'}\E|\cD_{r,t}^2M_\e'|^mdrdt\\
&\leq C\left(\e^m(T_\e')^{m}+\e^{2m-\rho}e^{m\lambda T_\e'}\right)\leq C\left(\e^m(T_\e')^{m}+\e^{m\kappa-\rho}\right),
\end{align*}
where we used \eqref{eq:T-prime-cond} in the last step. Choosing $\rho\in(0,m\kappa)$ to ensure that
the right-hand side of the previous display converges to zero, we finish the proof.
\end{proof}

\bibliographystyle{Martin}
\bibliography{citations}

\begin{thebibliography}{Nua06}
\expandafter\ifx\csname url\endcsname\relax
  \def\url#1{\texttt{#1}}\fi
\expandafter\ifx\csname urlprefix\endcsname\relax\def\urlprefix{URL }\fi
\expandafter\ifx\csname href\endcsname\relax
  \def\href#1#2{#2}\fi
\expandafter\ifx\csname burlalt\endcsname\relax
  \def\burlalt#1#2{\href{#2}{\texttt{#1}}}\fi

\bibitem[AB11]{AB2011}
\textsc{S.~Almada} and \textsc{Y.~Bakhtin}.
\newblock Normal forms approach to diffusion near hyperbolic equilibria.
\newblock \emph{Nonlinearity} \textbf{24}, (2011), 1883--1907.

\bibitem[Bak08]{Bak2008}
\textsc{Y.~Bakhtin}.
\newblock Exit asymptotics for small diffusion about an unstable equilibrium.
\newblock \emph{Stochastic Processes and their Applications} \textbf{118},
  no.~5, (2008), 839--851.

\bibitem[Bak10]{Bak2010}
\textsc{Y.~Bakhtin}.
\newblock Small noise limit for diffusions near heteroclinic networks.
\newblock \emph{Dynamical Systems} \textbf{25}, no.~3, (2010), 413--431.

\bibitem[Bak11]{Bak2011}
\textsc{Y.~Bakhtin}.
\newblock Noisy heteroclinic networks.
\newblock \emph{Probab. Theory Relat. Fields} \textbf{150}, (2011), 1--42.

\bibitem[Bas11]{Bass:MR2856623}
\textsc{R.~F. Bass}.
\newblock \emph{Stochastic processes}, vol.~33 of \emph{Cambridge Series in
  Statistical and Probabilistic Mathematics}.
\newblock Cambridge University Press, Cambridge, 2011,  xvi+390.

\bibitem[CV16]{Champagnat2016}
\textsc{N.~Champagnat} and \textsc{D.~Villemonais}.
\newblock Exponential convergence to quasi-stationary distribution and
  ${Q}$-process.
\newblock \emph{Probability Theory and Related Fields} \textbf{164}, no.~1,
  (2016), 243--283.
\newblock
  \burlalt{doi:10.1007/s00440-014-0611-7}{http://dx.doi.org/10.1007/s00440-014-0611-7}.

\bibitem[Day95]{Day95}
\textsc{M.~V. Day}.
\newblock On the exit law from saddle points.
\newblock \emph{Stochastic Process. Appl.} \textbf{60}, no.~2, (1995),
  287--311.
\newblock
  \burlalt{doi:10.1016/0304-4149(95)00063-1}{http://dx.doi.org/10.1016/0304-4149(95)00063-1}.

\bibitem[Eiz84]{Eizenberg:MR749377}
\textsc{A.~Eizenberg}.
\newblock The exit distributions for small random perturbations of dynamical
  systems with a repulsive type stationary point.
\newblock \emph{Stochastics} \textbf{12}, no. 3-4, (1984), 251--275.
\newblock
  \burlalt{doi:10.1080/17442508408833304}{http://dx.doi.org/10.1080/17442508408833304}.

\bibitem[FW12]{FW2012}
\textsc{M.~Freidlin} and \textsc{A.~Wentzell}.
\newblock \emph{Random Perturbations of Dynamical Systems}.
\newblock Grundlehren der mathematischen Wissenschaften. Springer, 2012.

\bibitem[Kif81]{Kifer1981}
\textsc{Y.~Kifer}.
\newblock The exit problem for small random perturbation of dynamical systems
  with a hyperbolic fixed point.
\newblock \emph{Israel J. Math.} \textbf{40}, no.~1, (1981), 74--96.

\bibitem[KS91]{KS1991}
\textsc{I.~Karatzas} and \textsc{S.~Shreve}.
\newblock \emph{Brownian Motion and Stochastic Calculus}.
\newblock Graduate Texts in Mathematics. Springer New York, 1991.

\bibitem[Mik95]{Mikami:MR1357028}
\textsc{T.~Mikami}.
\newblock Large deviations for the first exit time on small random
  perturbations of dynamical systems with a hyperbolic equilibrium point.
\newblock \emph{Hokkaido Math. J.} \textbf{24}, no.~3, (1995), 491--525.
\newblock
  \burlalt{doi:10.14492/hokmj/1380892606}{http://dx.doi.org/10.14492/hokmj/1380892606}.

\bibitem[Nua06]{Nualart2006}
\textsc{D.~Nualart}.
\newblock \emph{The {M}alliavin calculus and related topics}.
\newblock Probability and its Applications (New York). Springer-Verlag, Berlin,
  second ed., 2006,  xiv+382.

\bibitem[NV09]{Nourdin-Viens:MR2556018}
\textsc{I.~Nourdin} and \textsc{F.~G. Viens}.
\newblock Density formula and concentration inequalities with {M}alliavin
  calculus.
\newblock \emph{Electron. J. Probab.} \textbf{14}, (2009), no. 78, 2287--2309.
\newblock
  \burlalt{doi:10.1214/EJP.v14-707}{http://dx.doi.org/10.1214/EJP.v14-707}.

\end{thebibliography}

\end{document}